\documentclass[12pt]{article} 
\usepackage{xypic,amssymb,latexsym,amsfonts,amsthm,verbatim,amscd,ifthen,amsmath} 
\usepackage{color} 

\begin{document}
\title{Poset Resolutions and Lattice-Linear Monomial Ideals\\ \vspace{.1in} with an appendix by\\Alexandre Tchernev} 
\author{Timothy B.P. Clark}
\date{\today}
%\address{Department of Mathematics and Statistics, University at Albany, Albany, NY 12222}
%\email{tc7999@albany.edu}

\newtheorem{theorem}[equation]{Theorem} 
\newtheorem{lemma}[equation]{Lemma} 
\newtheorem{proposition}[equation]{Proposition} 
\newtheorem{corollary}[equation]{Corollary} 
\newtheorem{observation}[equation]{Observation} 
\newtheorem{conjecture}[equation]{Conjecture} 
\newtheorem{example}[equation]{Example} 
\newtheorem{definition}[equation]{Definition} 
\newtheorem{remark}[equation]{Remark} 
\newtheorem{remarks}[equation]{Remarks} 
\newtheorem{notation}[equation]{Notation} 
\newtheorem*{itheorem}{Theorem}
\newtheorem*{acknowledgments}{Acknowledgments}
\newenvironment{indented}{\begin{list}{}{}\item[]}{\end{list}} 

\renewcommand{\:}{\! :\ } 
\newcommand{\p}{\mathfrak p} 
\newcommand{\m}{\mathfrak m}
\newcommand{\g}{\gamma} 
\newcommand{\lra}{\longrightarrow} 
\newcommand{\ra}{\rightarrow} 
\newcommand{\altref}[1]{{\upshape(\ref{#1})}} 
\newcommand{\bfa}{\boldsymbol{\alpha}} 
\newcommand{\bfb}{\boldsymbol{\beta}} 
\newcommand{\bfg}{\boldsymbol{\gamma}} 
\newcommand{\bfd}{\boldsymbol{\delta}} 
\newcommand{\bfM}{\mathbf M} 
\newcommand{\bfN}{\mathbf N}
\newcommand{\bfI}{\mathbf I} 
\newcommand{\bfC}{\mathbf C} 
\newcommand{\bfB}{\mathbf B} 
\newcommand{\bfD}{\mathbf D}
\newcommand{\bfG}{\mathbf G}
\newcommand{\bsfC}{\bold{\mathsf C}} 
\newcommand{\bsfT}{\bold{\mathsf T}}
\newcommand{\mc}{\mathcal} 
\newcommand{\smsm}{\smallsetminus} 
\newcommand{\ol}{\overline} 
\newcommand{\twedge}
           {\smash{\overset{\mbox{}_{\circ}}
                           {\wedge}}\thinspace} 
\newcommand{\pring}{\Bbbk[x_1,\ldots,x_n]}
\newcommand{\irr}{(x_1,\ldots,x_n)}
\newcommand{\Z}{\textup{Z}}
\newcommand{\D}{\textbf{\textup{D}}}
\newcommand{\B}{\textup{B}}
\newcommand{\Po}{\mathcal{P}}
\newcommand{\La}{\mathcal{L}}
\newcommand{\Ho}{\widetilde{H}}
\newcommand{\scr}[1]{\mathfrak{#1}}
\newcommand{\pr}{\textup{proj}}
\newcommand{\sx}[1]{\textbf{\textup{G}}_{#1}}
\newcommand{\cplx}[1]{\Gamma_{#1}}
\newcommand{\ld}{\lessdot}
\newcommand{\rk}{\textup{rk}}
\newcommand{\crk}[2]{\rk(#2)-\rk(#1)}
\newcommand{\mbb}[1]{\mathbb{#1}}
\newcommand{\G}{\textbf{G}}
\newcommand{\ds}{\displaystyle}
\newcommand{\sgn}{\textup{sgn}}
\newcommand{\Sy}{\Sigma}
\newcommand{\sd}{\textbf{sd}}
\newcommand{\unsd}{\textbf{unsd}}
\newcommand{\str}[2]{\textup{St}(#1,#2)}
\newcommand{\CC}{\widetilde{\mathcal{C}}_\bullet}
\newcommand{\lcm}{\textup{lcm}}
\newcommand{\id}{\textup{id}}
\newcommand{\HH}{H} 

\newcommand{\mysection}[1]
{\section{#1}\setcounter{equation}{0}
             \numberwithin{equation}{section}}

\newcommand{\mysubsection}[1]
{\subsection{#1}\setcounter{equation}{0}
                \numberwithin{equation}{subsection}}

\newcommand{\mysubsubsection}[1]
{\subsubsection{#1}\setcounter{equation}{0}
                \numberwithin{equation}{subsubsection}}

\maketitle

\begin{abstract}
We introduce the class of lattice-linear 
monomial ideals and use the LCM-lattice to give an 
explicit construction for their minimal free 
resolution.  The class of lattice-linear 
ideals includes (among others) the class of monomial 
ideals with linear free resolution and the class of 
Scarf monomial ideals.  Our main tool is a new 
construction by Tchernev that produces from a 
map of posets $\eta:P\lra\mbb{N}^n$ a sequence 
of multigraded modules and maps.
\end{abstract}

\mysection{Introduction}
Let $R=\pring$ be a polynomial ring where $\Bbbk$ is a field, 
considered with its standard $\mbb{Z}^n$-grading (multigrading) 
and let $N$ be an ideal in $R$ generated by monomials.  

In \cite{GPW}, Gasharov, Peeva and Welker express 
the Betti numbers of $R/N$ using the homology groups of 
of the LCM-lattice $L_N$ of $N$.  They further show  
that the isomorphism class of $L_N$ determines the structure 
of the minimal free resolution of $R/N$.  Motivated by 
these results, we introduce the class of 
\emph{lattice-linear} monomial ideals.  
A lattice-linear ideal has the mapping structure of 
its minimal free resolution encoded in the covering 
relations of its LCM-lattice.  In our main result, 
Theorem \ref{LLT}, we construct explicitly the minimal 
free resolution of any lattice-linear ideal from its 
LCM-lattice.  

The class of lattice-linear ideals contains extensively studied subclasses,  
including the class of monomial ideals with a linear free resolution 
\cite{HH,HHMT,HRW,JW,MS} and the class of Scarf ideals \cite{BPS}.
For each of these two subclasses, minimal free resolutions have been 
constructed using different techniques which are also distinct from 
the one described in this paper.  

The key tool used to produce the lattice-linear resolutions 
is a new construction due to Tchernev which takes  
a partially ordered set (poset) $P$ as its input and produces 
a sequence of $\Bbbk$-vector spaces and $\Bbbk$-linear maps 
$C_\bullet(P)$.  Two variations of this construction are discussed, 
each based upon calculating the homology of the poset $P$.  
Under sufficiently desirable conditions (for example when the atoms 
of $P$ form a crosscut), the two variations are canonically 
isomorphic.  In general, the sequence $C_\bullet$ 
need not be a complex, nor need it be exact.  

When there exists a map of posets $\eta:P\lra\mbb{N}^n$, we 
homogenize the sequence $C_\bullet(P)$ to produce a 
sequence of multigraded modules and multigraded morphisms 
$\mbb{F}(\eta)$ which ``approximates'' a free resolution 
of the monomial ideal whose generators have their degrees given 
by the images of the atoms.  The poset map utilized 
in our main result to construct the minimal free resolution of 
a lattice-linear ideal is the degree map, $\deg:L_N\lra\mbb{N}^n$, 
which sends a monomial to its multidegree.  

\mysection{Poset Combinatorics}\label{Comb}
Let $(P,\le)$ be a finite poset with least 
element $\hat{0}$.  A totally ordered subset  $\sigma\subseteq{P}$ 
which has the form $\alpha_0<\cdots<\alpha_k$ is called a 
\emph{chain of length $k$}, and for $\alpha\in{P}$, the rank of an 
element $\alpha$ is 
$\rk(\alpha)=\sup\{l:\alpha_0<\cdots<\alpha_l=\alpha\}$.  
A subset of $P$ comprised entirely of elements which are pairwise 
incomparable is called an anti-chain.  An element $\beta\in{P}$ 
is \emph{covered} by $\alpha$ (which we write $\beta\ld\alpha$) 
when $\beta<\alpha$ and there exists no $\gamma\in{P}$ such 
that $\beta<\gamma<\alpha$.  An open interval in $P$ is denoted 
$(\beta,\alpha)=\{\sigma:\beta<\alpha_i<\alpha \textrm{ for all $\alpha_i\in\sigma$ }\}$, 
with closed and half-open intervals denoted similarly.  
Recall that the set of atoms of $P$ is $$S=\{a\in{P}:\hat{0}\ld{a}\},$$ 
and setting $\rk(\hat{0})=0$, we have $\rk(a)=1$ for every $a\in{S}$.  
The poset $P$ is said to be \emph{ranked} when 
$\rk(\alpha)=\rk(\beta)+1$ for every $\beta\ld{\alpha}\in{P}$.  
When they exist, the meet (greatest lower bound) and join 
(least upper bound) of a subset $A\subseteq{P}$ are 
denoted as $\wedge A$ and $\vee A$, respectively.  

For the purpose of topological analysis of $P$, recall that the order 
complex of a poset is the simplicial complex $\Delta(P)$ whose 
$k$-dimensional faces are in one-to-one correspondence with the length 
$k$ chains of $P$.  As is standard, whenever discussing topological 
properties of $P$, we are implicitly referring to the topological 
properties of $\Delta(P)$.  Another topological object of interest 
is the the crosscut simplicial complex $\Gamma(P,C)$ associated to $P$.  
Recall that a set $C\subset{P}$ is called a crosscut if it satisfies 
the following three properties; 
\begin{enumerate}
\item $C$ is an anti-chain, 
\item For every finite chain $\sigma$ in $P$ there exists some element 
      in $C$ which is comparable to each element in $\sigma$,
\item If $A\subseteq{C}$ is bounded in $P$, then either the join 
      $\vee{A}$ or the meet $\wedge{A}$ exists in $P$.  
\end{enumerate}
For a crosscut $C$, the crosscut simplicial complex $\Gamma(P,C)$ is 
defined as the collection of all subsets of $C$ which are bounded in $P$.  
We introduce a family of simplicial complexes indexed by 
the elements of $P$ as follows.

\begin{definition} 
$\phantom{3}$
\begin{enumerate}
\item For $\lambda\in{P}$, set $\sx{\lambda}$ to be the full 
simplex on the set $$S_\lambda=\{a\in{S}:a\le{\lambda}\}$$ 
and for $\alpha\in P$ set $$\cplx{\alpha}=\bigcup_{\lambda\ld\alpha}\sx{\lambda}.$$  
\item For each $\lambda\ne\beta\ld\alpha$ we have 
$\sx{\lambda}\cap\sx{\beta}\subset\cplx{\lambda}$ 
and decompose $\cplx{\alpha}$ as 
$$
\cplx{\alpha}
=\sx{\lambda}\cup\left(\bigcup_{\stackrel{\beta\ld\alpha}{\lambda\ne\beta}}\sx{\beta}\right).
$$  
For a fixed $\lambda\ld\alpha$ set 
$$
\cplx{{\alpha,\lambda}}
=\sx{\lambda}\cap\left(\bigcup_{\stackrel{\beta\ld\alpha}{\lambda\ne\beta}}\sx{\beta}\right)
\subset\cplx{\lambda}.
$$  
\end{enumerate}
\end{definition}

The order complex of an open interval 
$\Delta_\alpha:=\Delta(\hat{0},\alpha)$ 
may be realized in a similar way.  

\begin{definition}
$\phantom{3}$
\begin{enumerate}
\item For $\alpha\in{P}$, set $\D_\alpha=\Delta(\hat{0},\alpha]$, 
so that $$\Delta_\alpha=\bigcup_{\lambda\ld\alpha}\D_\lambda.$$  
\item For each $\lambda\ne\beta\ld\alpha$ we have 
$\D_\lambda\cap{\D_\beta}\subset\Delta_\lambda$ 
and decompose $\Delta_\alpha$ as 
$$
\Delta_\alpha=
\D_\lambda\cup\left(\bigcup_{\stackrel{\beta\ld\alpha}{\lambda\ne\beta}}\D_\beta\right),$$ 
and set 
$$\Delta_{\alpha,\lambda}=
\D_{\lambda}\cap\left(\bigcup_{\stackrel{\beta\ld\alpha}{\lambda\ne\beta}}
\D_\beta\right)\subset\Delta_\lambda.$$
\end{enumerate}
\end{definition}
 
It is advantageous to consider the families of 
simplicial complexes $\{\Delta_\alpha:\alpha\in P\}$ 
and $\{\Gamma_\alpha:\alpha\in P\}$ when analyzing 
the topology of $P$, as one family may have structural 
advantages over the other.  

\begin{remark}\label{CC}
If the set of atoms $S$ forms a crosscut in $P$, then 
the sets $S_\lambda$ for $\lambda\ld\alpha$ are the 
maximal subsets of $S$ which are bounded in the open interval 
$(\hat{0},\alpha)$ of $P$.  Each simplicial complex 
$\cplx{\alpha}$ is therefore identical to the crosscut 
complex $\Gamma((\hat{0},\alpha),S_\alpha)$.  Invoking 
the Crosscut Theorem (\cite[Theorem 10.8]{BjBrick}), 
then $\Delta_\alpha$ and $\cplx{\alpha}$ 
are homotopic for every $\hat{0}\ne\alpha\in{P}$.
\end{remark}
%\mysubsection{Utilizing the order complex}\label{utilOrder}
Turning our attention to the family of simplicial 
complexes $$\{\Delta_\alpha:\alpha\in{P}\}$$ we 
describe a sequence of vector spaces and vector space maps 
$$C_\bullet(P):\cdots\lra{C_i}
\stackrel{\varphi_i}{\lra}C_{i-1}\lra
\cdots\lra{C_1}\stackrel{\varphi_1}{\lra}C_0$$ 
whose structure is determined by the simplicial 
complexes $\Delta_\alpha$.  

\begin{definition}
$\phantom{3}$ 
\begin{enumerate}
\item Set $C_0=C_0(P)=\Ho_{-1}(\emptyset,\Bbbk)\cong\Bbbk$.
\item For $i\ge{1}$, set $C_{i,\alpha}=C_{i,\alpha}(P)=\Ho_{i-2}(\Delta_\alpha,\Bbbk)$ and 
$$C_i=C_i(P)=\ds \bigoplus_{\alpha\in{P}\smsm\{\hat{0}\}}C_{i,\alpha}.$$
\end{enumerate}
\end{definition}

\begin{remark}
When $i=1$ and $\alpha\in{S}$, we have $\Delta_\alpha=\D_{\hat{0}}=\emptyset$
and thus, $C_{1,\alpha}=\Ho_{-1}(\emptyset,\Bbbk)\cong\Bbbk$.  
If $i=1$ and $\alpha\notin{S}$ then 
$\Delta_\alpha=\ds \bigcup_{\lambda\ld\alpha}\D_\lambda\ne\emptyset$ and hence $C_{1,\alpha}=\Ho_{-1}(\Delta_\alpha,\Bbbk)=0$.  
Therefore, 
  $$\ds
    C_1=
    \bigoplus_{\alpha\in{S}}C_{1,\alpha}=
    \bigoplus_{\alpha\in{S}}\Ho_{-1}(\emptyset,\Bbbk)\cong
    \bigoplus_{\alpha\in{S}}\Bbbk.
  $$
\end{remark}

Next, given $\lambda\ld\alpha$, we consider the Mayer-Vietoris sequence 
in reduced homology for the triple 
  $$
  \ds \left(
  \D_\lambda,\hspace{.1in}
  \bigcup_{\stackrel{\beta\ld\alpha}{\lambda\ne\beta}}\D_\beta,\hspace{.1in}
  \Delta_\alpha
  \right).
  $$  
We write $\iota:\Ho_{i-3}(\Delta_{\alpha,\lambda},\Bbbk)\lra\Ho_{i-3}(\Delta_\lambda,\Bbbk)$ 
for the map induced in homology by the inclusion map and
$$
\delta_{i-2}^{\alpha,\lambda}:
\Ho_{i-2}(\Delta_\alpha,\Bbbk)\lra
\Ho_{i-3}(\Delta_{\alpha,\lambda},\Bbbk)
$$ 
for the connecting homomorphism from the Mayer-Vietoris sequence.  
Recall this homomorphism takes the class 
$[c]\in\Ho_{i-2}(\Delta_\alpha,\Bbbk)$ to the class 
$[d_{i-2}(c')]\in\Ho_{i-3}(\Delta_{\alpha,\lambda},\Bbbk)$ where 
$c'+c''=c\in\widetilde{\mathcal{C}}_{i-2}(\Delta_\alpha,\Bbbk)$, 
and $c'$ and $c''$ are any components of $c$ that are 
supported by $\D_\lambda$ and 
$\ds\bigcup_{\lambda\ne\mu\ld\alpha}\D_\mu$ 
respectively.  Here, $d$ is the usual boundary map.    

We can now proceed with the definition of the maps 
$\varphi_{i}:C_i\lra{C_{i-1}}$ for the sequence $C_\bullet(P)$. 

\begin{definition}
$\phantom{3}$
\begin{enumerate}
\item Define $\varphi_1:C_1\lra{C_0}$ componentwise by  
$\varphi_1\vert_{C_{1,\alpha}}=\id_{\Ho_{-1}(\emptyset,\Bbbk)}$.
\item For $i\ge 2$ define $\varphi_i:C_i\lra C_{i-1}$ componentwise by   $$\varphi_{i}\vert_{C_{i,\alpha}}=\sum_{\lambda\ld\alpha}\varphi_{i}^{\alpha,\lambda}$$
where $$\varphi_{i}^{\alpha,\lambda}:C_{i,\alpha}\lra{C_{{i-1},\lambda}}$$ 
is the composition $\varphi_{i}^{\alpha,\lambda}=\iota\circ\delta_{i-2}^{\alpha,\lambda}$.
\end{enumerate}
\end{definition}

%\mysubsection{Utilizing the complexes $\Gamma_\alpha$}
The construction above may be performed using the family of subcomplexes 
$$\{\Gamma_\alpha:\alpha\in{P}\}.$$  A priori, the sequence $C_\bullet(P,\Gamma)$ 
obtained in this way is different from the the sequence $C_\bullet(P)$ 
based on the family of subcomplexes $$\{\Delta_\alpha:\alpha\in{P}\}.$$  
However, as the meaning will always be clear from the context, we 
will use the same notation for the components and the 
same notation for the maps in both sequences.  To be specific, we 
describe the sequence of vector spaces and vector space maps 
$$
C_\bullet(P,\Gamma):\cdots\lra{C_i}
\stackrel{\varphi_i}{\lra}C_{i-1}\lra
\cdots\lra{C_1}\stackrel{\varphi_1}{\lra}C_0
$$
as follows.  

\begin{definition}
$\phantom{3}$  
\begin{enumerate}
\item Set $C_0=C_0(P,\Gamma)=\Ho_{-1}(\emptyset,\Bbbk)\cong\Bbbk$.  
\item For $i\ge{1}$, set  
$C_{i,\alpha}=C_{i,\alpha}(P,\Gamma)=\Ho_{i-2}(\cplx{\alpha},\Bbbk)$ and 
$$C_i=C_i(P,\Gamma)=\ds \bigoplus_{\alpha\in{P}\smsm\{\hat{0}\}}C_{i,\alpha}.$$
\end{enumerate}
\end{definition}

\begin{remark}
Again, when $i=1$ and $\alpha\in{S}$, we have $\cplx{\alpha}=\sx{\hat{0}}=\emptyset$
and thus, $C_{1,\alpha}=\Ho_{-1}(\emptyset,\Bbbk)\cong\Bbbk$.  
If $i=1$ and $\alpha\notin{S}$ then 
$\cplx{\alpha}=\ds \bigcup_{\lambda\ld\alpha}\sx{\lambda}\ne\emptyset$ and hence $C_{1,\alpha}=\Ho_{-1}(\cplx{\alpha},\Bbbk)=0$.  
Therefore, 
  $$\ds
    C_1=
    \bigoplus_{\alpha\in{S}}C_{1,\alpha}=
    \bigoplus_{\alpha\in{S}}\Ho_{-1}(\emptyset,\Bbbk)\cong
    \bigoplus_{\alpha\in{S}}\Bbbk.
  $$
\end{remark}
  
As before, given $\lambda\ld\alpha$, we consider the Mayer-Vietoris 
sequence in reduced homology for the triple 
  $$\ds
  \left(\sx{\lambda},\hspace{.1in}
  \bigcup_{\lambda\ne\beta\ld\alpha}\sx{\beta},\hspace{.1in}
  \cplx{\alpha}\right).
  $$  
We set $\iota:\Ho_{i-3}(\cplx{{\alpha,\lambda}},\Bbbk)\lra\Ho_{i-3}(\cplx{\lambda},\Bbbk)$ 
to be the map induced in homology by the inclusion map and
  $$\delta_{i-2}^{\alpha,\lambda}:
  \Ho_{i-2}(\cplx{\alpha},\Bbbk)\lra
  \Ho_{i-3}(\cplx{{\alpha,\lambda}},\Bbbk)$$ 
to be the connecting homomorphism from the Mayer-Vietoris sequence.  
%This detail may eventually be commented out.
This homomorphism takes the class $[c]\in\Ho_{i-2}(\cplx{\alpha},\Bbbk)$ to the class $[d_{i-2}(c')]\in\Ho_{i-3}(\cplx{{\alpha,\lambda}},\Bbbk)$ 
where $c'+c''=c\in\mathcal{C}_{i-2}(\cplx{\alpha},\Bbbk)$, 
and $c'$ and $c''$ are components of $c$ that are supported by $\sx{\lambda}$ 
and by $\ds\bigcup_{\lambda\ne\beta\ld\alpha}\sx{\mu}$ respectively.  
Again, $d$ is the usual boundary map.  

\begin{definition}
$\phantom{3}$
\begin{enumerate}
\item Define $\varphi_1:C_1\lra{C_0}$ componentwise by  
$\varphi_1\vert_{C_{1,\alpha}}=\id_{\Ho_{-1}(\emptyset,\Bbbk)}$. 
\item For $i\ge 2$ define $\varphi_i:C_i\lra C_{i-1}$ componentwise by   $$\varphi_{i}\vert_{C_{i,\alpha}}=\sum_{\lambda\ld\alpha}\varphi_{i}^{\alpha,\lambda}$$
where $$\varphi_{i}^{\alpha,\lambda}:C_{i,\alpha}\lra{C_{{i-1},\lambda}}$$ 
is the composition $\varphi_{i}^{\alpha,\lambda}=\iota\circ\delta_{i-2}^{\alpha,\lambda}$. 
\end{enumerate}
\end{definition}

\section{Properties of $C_\bullet(P)$ and $C_\bullet(P,\Gamma)$}\label{Cproperties}
The sequences $C_\bullet(P)$ and $C_\bullet(P,\Gamma)$ 
are not necessarily complexes of vector spaces, and even if 
one is a complex, it need not be exact.  While necessary 
and sufficient conditions for the construction to produce an exact 
complex are not known, if the poset is ranked, $C_\bullet(P)$ is 
a complex of vector spaces (Proposition \ref{itsacomplex} in the appendix).  
Depending on the structure of each of 
the collections, it may be more advantageous to use one instead 
of the other.  When the poset has sufficiently good structure, we prove

\begin{proposition}\label{CCdiagram} Suppose that the set of atoms $S$ forms a 
crosscut in $P$, and let $f:P\lra\Gamma(P,S)$ be the 
homotopy equivalence given in \cite{BjBrick}.  Then for 
every $\beta\ld\alpha\in{P}$ and every $i\ge 1$ the map $f$ 
induces canonically a commutative diagram 
\begin{displaymath}
\begin{CD}
\Ho_i(\Delta_\alpha,\Bbbk) @>\mathfrak{f}_i>> \Ho_i(\cplx{\alpha},\Bbbk) \\
@V \varphi_i VV @V \varphi_i VV \\
\Ho_{i-1}(\Delta_\beta,\Bbbk) @>\mathfrak{f}_{i-1}>> \Ho_{i-1}(\cplx{\beta},\Bbbk) \\
\end{CD}
\end{displaymath}
where $\mathfrak{f}_*$ is an isomorphism in homology.
\end{proposition}

\begin{corollary}\label{LatticeCor}
If $P$ is a lattice or a geometric semilattice, 
then $C_\bullet(P)$ and $C_\bullet(P,\Gamma)$ 
are canonically isomorphic.
\end{corollary}

\begin{proof}
The set of atoms of $P$ forms a crosscut.
\end{proof}

\begin{proof}[Proof of Proposition \ref{CCdiagram}]
Recall that the barycentric subdivision of a simplicial complex $\Omega$ 
may be realized as the simplicial complex $\sd(\Omega)=\Delta(P(\Omega))$.
Write $$P_{\ge{x}}=\{y\in{P}:y\ge{x}\}$$ and suppose that $\sigma$ 
is a chain of $P$.  An order reversing map of posets 
$g:P(\Delta(P))\lra P(\Gamma(P,S))$ is defined in \cite{BjBrick} as 
$$\sigma\mapsto\left\{x\in{S}:\sigma\in\Delta({P_{\ge{x}}})\right\}.$$  
This map induces in the usual way a simplicial map on the corresponding 
order complexes $h:\sd(\Delta(P))\lra\sd(\Gamma(P,S))$ which gives 
rise to a chain map $$h_\sharp:\CC(\sd(\Delta(P)))\lra\CC(\sd(\Gamma(P,S))).$$

The homotopy equivalence $f$ in \cite{BjBrick} can be defined on the level 
of chains as $f_\sharp=\unsd_\sharp\circ h_\sharp\circ\sd_\sharp$ where 
$$
\CC(\Delta(P))
\stackrel{\sd_\sharp}{\lra}
\CC(\sd(\Delta(P)))
\stackrel{h_\sharp}{\lra}
\CC(\sd(\Gamma(P,S)))
\stackrel{\unsd_\sharp}{\lra}
\CC(\Gamma(P,S)),
$$
the simplicial map $\unsd$ is defined by fixing a total ordering on the set 
of atoms $S$ and sending $S_a=\{x\in{S}\mid x\le a\}$ to its minimum 
element $s_a=\min S_a$ under this ordering; and $\unsd_\sharp$ is the 
induced chain map.  Recall that the \emph{star} of a vertex $v$ in a 
simplicial complex $K$, denoted $\str{v}{K}$, is the union of the 
interiors of the simplices in $K$ that have $v$ as a vertex.  Given 
any vertex $S_a$ of $\sd(\Gamma(P,S))$ then the vertex $s_a$ of 
$\Gamma(P,S)$ satisfies 
$$
\str{S_a}{\sd(\Gamma(P,S)}\subset\str{s_a}{\Gamma(P,S)}.
$$
Invoking \cite[Lemma 15.1]{Mun}, it follows that $\unsd$ is a 
simplicial approximation to the identity, and therefore 
via the algebraic subdivision theorem \cite[Theorem 17.2]{Mun} 
$\unsd_\sharp$ is a chain map that is a homotopy inverse to the 
subdivision map $\Gamma(P,S)\lra\sd(\Gamma(P,S))$.

Next, recall that under barycentric subdivision a face 
$\sigma=\{a_0,\ldots,a_k\}\in\Delta(P)$ with $a_0<\cdots<a_k\in{P}$ 
has image 
\begin{equation}\label{subdiv}
\sd_\sharp(\sigma)=
\sum_{\rho\in\Sy_{k+1}} 
\varepsilon_\rho
\left\{
\{a_{\rho(k)}\},\{a_{\rho(k-1)},a_{\rho(k)}\},\cdots,\{a_{\rho(0)},\ldots,a_{\rho(k)}\}
\right\}
\end{equation}
where $\Sy_{k+1}$ is the group 
of permutations on the set $\{0,1,\ldots,k\}$ and $\varepsilon_\rho$ 
denotes the sign of the permutation $\rho$.  
Applying $h$ to the chain 
$$
\left\{
\{a_{\rho(k)}\},
\{a_{\rho(k-1)},a_{\rho(k)}\},
\cdots,
\{a_{\rho(0)},\ldots,a_{\rho(k)}\}
\right\}
$$
yields the face
$$
\left\{
S_{a_{\rho(k)}},
S_{a_{\rho(k-1)}},
\cdots,
S_{a_{\rho(0)}}
\right\}.
$$
Unless $\rho$ is the identity of $\Sy_{k+1}$, 
this face has dimension less than or equal to $k-1$.  
It follows that under the chain map $h_\sharp$, 
the sum in equation (\ref{subdiv}) has image 
\begin{equation}\label{survterm}
h_\sharp(\sd_\sharp(\sigma))=
 \left\{
         S_{a_k},S_{a_{k-1}},\cdots,S_{a_0}
 \right\}.
\end{equation}
Considering the simplicial map 
$\unsd:\sd(\Gamma(P,C))\lra\Gamma(P,C)$, 
if $s_{a_t}=s_{a_{t-1}}$ for some $1\le t\le k$, then the face 
$$
\left\{s_{a_k},\cdots,s_{a_0}\right\}=
\unsd\left(\{S_{a_k},S_{a_{k-1}},\cdots,S_{a_0}\}\right)
$$ 
has dimension less than or equal to $k$.  Thus, 
$$\unsd_\sharp\left(\left\{S_{a_k},S_{a_{k-1}},\cdots,S_{a_0}\right\}\right)=0$$
except when $s_{a_k}<s_{a_{k-1}}<\cdots<s_{a_0}$.  
Therefore, 
\begin{displaymath}
f_\sharp(\sigma)=
\unsd_\sharp(h_\sharp(\sd_\sharp(\sigma)))= 
           \left\{ 
           \begin{array}{cl}
           \left\{s_{a_k},\cdots,s_{a_0}\right\} 
           & 
           \textrm{when } s_{a_k}<\cdots<s_{a_0} \\
           0 & \textrm{otherwise}
           \end{array}
           \right.
\end{displaymath}
It is clear from this description that for each $\alpha$ 
and each $\sigma\in\Delta(\hat{0},\alpha]$ 
we have $f_\sharp(\sigma)\in\CC(\sx{\alpha})$ and therefore 
the horizontal maps in our diagram are well defined and 
are isomorphisms by Remark \ref{CC}.
%If $\sigma$ is a face with the property that 

We now turn to proving the commutativity of the diagram.
Fix $\beta\ld\alpha$, suppose that $[\pi]$ is a homology 
class in $\Ho_i(\Delta_\alpha)$ and write 
\begin{eqnarray}\label{sepsum}
\pi & = & \sum_{\dim(\sigma)=i} 
      c_\sigma\cdot\sigma\\
& = & \sum_{a^\sigma_i\le\beta} 
      c_\sigma\cdot\sigma
      +
      \sum_{a^\sigma_i\not\le\beta} 
      c_\sigma\cdot\sigma.\nonumber 
\end{eqnarray}
for a representative of this class where 
$\sigma=\{a^\sigma_0,\ldots,a^\sigma_i\}$ is
oriented by $a^\sigma_0<\cdots<a^\sigma_i$ 
and $c_\sigma\in\Bbbk$ is a scalar.  
Applying the isomorphism $\mathfrak{f}$, we have 
\begin{eqnarray*}
\mathfrak{f}([\pi]) & = & [f_\sharp(\pi)]\\
& = & \left[
      \sum_{\dim(\sigma)=i} 
      c_\sigma\cdot f_\sharp(\sigma)
      \right]\\
& = & \left[
      \sum_{a^\sigma_i\le\beta}
      c_\sigma\cdot f_\sharp(\sigma)
      +
      \sum_{a^\sigma_i\not\le\beta} 
      c_\sigma\cdot f_\sharp(\sigma)
      \right].
\end{eqnarray*}
Since the terms appearing in 
the first summand are faces in $\sx{\beta}$ 
and the terms appearing in the second summand are
faces in $$\ds \bigcup_{\beta\ne\gamma\ld\alpha}\sx{\gamma},$$
then applying the map $\varphi_i$
to $[f_\sharp(\pi)]$ yields
$$
\varphi_i(\mathfrak{f}([\pi]))=
\left[
\sum_{a_i^\sigma\le\beta}
c_\sigma\cdot d_i\circ f_\sharp(\sigma)
\right].
$$
Again taking $[\pi]\in\Ho_i(\Delta_\alpha)$,
we apply $\varphi_i$ and achieve the homology class
$$
\varphi_i([\pi])=
\left[
\sum_{a_i^\sigma\le\beta} c_\sigma\cdot d_i(\sigma)
\right].
$$
%Any term in the sum (\ref{sepsum}) which has $a_i^\sigma<\beta$ does not contribute to the homology class above since such a term is contained entirely within the face $\D_\beta$ and is therefore the boundary of of some subcomplex of $\D_\beta$. 
Lastly, we apply the isomorphism $\mathfrak{f}$ to achieve 
\begin{eqnarray*}
\mathfrak{f}(\varphi_i([\pi])) & = & \left[
\sum_{a_i^\sigma\le\beta}
c_\sigma\cdot f_\sharp\circ d_i(\sigma)
\right]\\
& = & 
\left[
\sum_{a_i^\sigma\le\beta}
c_\sigma\cdot d_i\circ f_\sharp(\sigma)
\right]\\
& = & 
\varphi_i([f_\sharp(\pi)]))
\end{eqnarray*}
so that the proposition is proven.
%since $f_\sharp$ is a chain map and therefore commutes with the differential $d_i$.
\end{proof}

\mysection{Poset Resolutions}\label{PosetRes}
Let $R=\pring$, let $\m=\irr$ be the unique graded 
maximal ideal of $R$, and 
$x^\alpha=x_{1}^{\alpha_1}\cdots{x_{n}^{\alpha_n}}$.  
We appeal to the standard $\mbb{Z}^n$-grading (mulitigrading) 
of $R$ and use the notation of Section \ref{Comb} for ordering 
in the partially ordered set $\mbb{Z}^n$.  We use the degree 
map $x^\alpha\mapsto\alpha$ and identify 
the monomials in $R$ with the elements of 
$\mbb{N}^n\subset\mathbb{Z}^n$. 

Suppose that $\eta:P\lra\mbb{N}^n$ is a map of partially 
ordered sets, and $S$ is the set of atoms of $P$.  Let $N$ be 
the ideal in $R$ generated by the monomials $$G(N)=\{x^{\eta(a)}:a\in{S}\}.$$  
The complex of vector spaces $C_\bullet(P)$ constructed in Section \ref{Comb} 
and associated to $P$ is homogenized using the map $\eta$ to 
produce 
$$
\mbb{F}=\mbb{F}(\eta):
  \cdots
  \lra 
  F_t
  \stackrel{\partial_t}{\lra} 
  F_{t-1}
  \lra
  \cdots
  \lra
  F_1
  \stackrel{\partial_1}{\lra} 
  F_0,
$$
a sequence of free multigraded 
$R$-modules and multigraded $R$-module homomorphisms which 
approximates a free resolution of the multigraded module $R/N$.  
This homogenization is carried out by constructing 
$F_0=R\otimes_{\Bbbk}C_0$ and grading the result 
with $\deg(x^\alpha\otimes{v})=\alpha$ for each $v\in C_0$.  
Similarly, for $i\ge{1}$, we set 
$$
\ds F_i=
\bigoplus_{{\hat{0}}\ne\lambda\in P}F_{i,\lambda}=
\bigoplus_{{\hat{0}}\ne\lambda\in P}R\otimes_{\Bbbk}C_{i,\lambda}
$$ 
where the grading is defined as 
$\deg(x^\alpha\otimes{v})=\alpha+\eta(\lambda)$ for each 
$v\in C_{i,\lambda}$.  The differential in this sequence 
of multigraded modules is defined 
componentwise in homological degree 1 as  
$$
\partial_1\arrowvert_{F_{1,\lambda}}
=x^{\eta(\lambda)}\otimes\varphi_1\arrowvert_{C_{1,\lambda}}
$$ 
and for $i\ge{1}$, the map 
$\partial_i:F_i\lra F_{i-1}$ is 
defined as 
$$
\ds \partial_i\arrowvert_{F_{i,\alpha}}
=\sum_{\lambda\ld\alpha}\partial^{\alpha,\lambda}_i
$$ 
where 
$\partial^{\alpha,\lambda}_i:F_{i,\alpha}\lra{F_{i-1,\lambda}}$ 
takes the form 
$\partial^{\alpha,\lambda}_i=x^{\alpha-\eta(\lambda)}\otimes\varphi_{i}^{\alpha,\lambda}$ 
for $\lambda\ld\alpha$.  

We are now in a position to make 
\begin{definition}
If\/ $\mbb{F}(\eta)$ is an acyclic complex of 
multigraded modules, then we say that it is 
a \emph{poset resolution} of the ideal $N$.   
\end{definition}

\begin{example}\label{Taylor} 
$\phantom{3}$
\begin{enumerate}
\item The Taylor resolution \cite{T} can be realized 
as a poset resolution where $P=\mathcal{B}_r$, 
the Boolean lattice.  The map 
$$\eta:\mathcal{B}_r\lra\mbb{N}^n$$ is defined on a 
lattice element $I\subseteq\{1,\ldots,r\}$ via 
$I\mapsto\deg(m_I)$, where $m_I=\lcm(m_i\in{G(N)}:i\in{I})$.
\item An arbitrary \emph{stable} monomial ideal is shown to admit
a minimal poset resolution in \cite{Clark} using a poset of 
Eliahou-Kervaire admissible symbols.  
\end{enumerate}
\end{example}

We conclude this section with the following general remark
on notation.  Any sequence $\mbb{P}$ of morphisms of 
free multigraded modules can be decomposed as 
$$\mbb{P}=\bigoplus_{\alpha\in{\mathbb{Z}^n}}\mbb{P}_\alpha$$ 
where each $\mbb{P}_\alpha$ is a sequence of maps of vector spaces, 
and is called the \emph{(multigraded) strand of \/ $\mbb{P}$ in degree $\alpha$}.  
We denote by $(\mbb{P}_\alpha)_i$ the $i^{th}$ component of the sequence
$\mbb{P}_\alpha$.  Using the multigrading, it is clear that
$$
(\m\mbb{P})_\alpha=
\sum_{\beta\ld\alpha}x^{\alpha-\beta}{\mbb{P}_\beta}\subset{\mbb{P}_{\alpha}}.
$$  
Further, we will identify $x^{\alpha-\beta}{\mbb{P}_\beta}$ 
with $\mbb{P}_\beta$ so that we may write 
$\mbb{P}_{\beta}\subset{\mbb{P}_\alpha}$ for $\beta\ld\alpha$.  
This allows us to consider 
$\mbb{P}_{\gamma}\subset{\mbb{P}_{\alpha}}$ for all $\gamma<\alpha$.  
In addition, we may now write
$$
(\m\mbb{P})_\alpha
=\sum_{\beta\ld\alpha}{\mbb{P}_\beta}
=\sum_{\gamma<\alpha}{\mbb{P}_\gamma}\subset{\mbb{P}_{\alpha}}.
$$  

\mysection{Lattice-Linear Monomial Ideals}\label{LL}
We now turn to the class of ideals that are the focus 
of this paper.  Recall that the \emph{LCM-lattice} 
associated to a monomial ideal $N$ is the
the set $L_N$ of least common multiples of the subsets of 
the set of minimal generators of $N$ (where by convention, 
$1$ is considered to be the least common 
multiple of the empty set) and ordering in $L_N$ is 
given by divisibility.  Recall that we identify monomials 
with their degree in $\mbb{N}^n$.  In particular, we consider 
$L_N$ as a sublattice of $\mbb{N}^n$.  As an immediate consequence 
of \cite[Theorem 3.1a]{BHres}, if the $i^{th}$ multigraded 
Betti number $\beta_{i,\alpha}(R/N)\ne 0$, then $\alpha\in L_N$.  
In particular, this means that if $B_i$ is any multihomogeneous 
basis of the free module $F_i$ in the minimal free resolution 
$\mbb{F}$ of $R/N$ then $\deg(v)\in L_N$ for each $v\in B_i$. 

\begin{definition}\label{LLDef}
Let $\mbb{F}$ be a minimal multigraded free resolution of $R/N$.
We say that $N$ is \emph{lattice-linear} if multigraded
bases $B_k$ of $F_k$ can be fixed for all $k$ 
so that for any $i\ge 1$ and any $v\in{B_i}$ the differential 
$$\partial^{\mbb{F}}(v)=\sum_{v'\in{B_{i-1}}}m_{v,v'}\cdot{v'}$$ 
has the property that if the coefficient $m_{v,v'}$ is 
nonzero then $\deg(v')\ld\deg(v)\in{L_N}$.
\end{definition}

\begin{remark}
The notion of lattice-linearity is dependent upon 
the characteristic of the ground field $\Bbbk$.  
For example, the ideal 
\begin{eqnarray*}
N & = & \langle x_1x_2x_3, x_1x_3x_5, x_1x_4x_5, x_2x_3x_4, x_2x_4x_5,\\
  &   & \phantom{f}x_1x_2x_6, x_1x_4x_6, x_2x_5x_6, x_3x_4x_6, x_3x_5x_6 \rangle
\end{eqnarray*}
in $R=\Bbbk[x_1,x_2,x_3,x_4,x_5,x_6]$ is lattice-linear 
if and only if $\textup{char}(\Bbbk)\ne2$.  
\end{remark}

We now state our main result.  

\begin{theorem}\label{LLT}
Define the map $\deg:L_N\lra\mathbb{N}^n$ by sending a 
monomial $m\in{L_N}$ to its degree 
$\deg(m)=(\alpha_1,\ldots,\alpha_n)\in\mathbb{N}^n$.  
Then the monomial ideal $N$ is lattice-linear 
if and only if $\mbb{F}(\deg)$ is its minimal free resolution.  
\end{theorem}

\begin{comment}
\begin{corollary}\label{linear}
Every monomial ideal which has a linear minimal free 
resolution is a lattice-linear monomial ideal.
\end{corollary}

\begin{proof}
Suppose that $N$ is an ideal with linear resolution 
and aiming for a contradiction, that $N$ is not lattice-linear.  
Then there exists $i>0$ and $e\in F_i$ of degree 
$\alpha$ such that in the expansion of $\partial^{\mbb{F}}(e)$ 
the element $e'\in F_{i-1}$ has multidegree $\beta$ 
which is not covered by $\alpha\in{L_N}$.  Therefore, 
there exists $\gamma\in{L_N}$ such that $\beta<\gamma<\alpha$.  
However, since $N$ has a linear resolution, 
$\deg(\alpha)=\deg(\beta)+1$, and there can be no 
multidegree $\gamma$ which fits this criteria for 
comparibility in $L_N$.  Hence, $N$ is lattice-linear.
\end{proof}
\end{comment}

We postpone the proof of Theorem \ref{LLT} in favor of 
two examples of lattice-linear ideals.  

First, the class of lattice-linear ideals clearly contains those ideals 
with a linear free resolution, whose minimal free resolutions 
have been constructed in \cite{JW} using tools from Discrete 
Morse Theory.  Our methods allow us to provide a considerably 
simpler and more transparent approach to constructing these 
minimal free resolutions.  

For our second example we recall from \cite{BPS} the Scarf 
simplicial complex 
$$
\Delta_N=
\{I\subseteq\{1,\ldots,r\}
\mid
{m_I}\ne{m_J}\textup{ for all } J\subseteq\{1,\ldots,r\} \textup{ other than } I\}.
$$  
When $I\in\Delta_N$, then $I$ is uniquely determined 
from $m_I$ as the set $I=\{i\mid{m_i}<m_I\}$.  
The ideal $N$ is called a \emph{Scarf} ideal if 
its minimal free resolution is supported on $\Delta_N$.  
For instance, the so-called \emph{generic} 
\cite{BPS,MS} ideals are Scarf.  Note that when $N$ is 
Scarf, the differential in its minimal free resolution 
takes the unique basis element $e_I$ 
labeled by the monomial $m_I$ to 
$$
\sum_{j=1}^{|I|}(-1)^{j+1}
\frac{m_I}{m_{I\smsm\{i_j\}}}
\cdot
{e_{I\smsm\{i_j\}}}
$$ where $I=\{i_1,\ldots,i_{|I|}\}$.  

\begin{proposition}\label{scarf}
Every Scarf ideal is a lattice-linear monomial ideal.  
In particular, for $\eta:P(\Delta_N)\lra\mbb{N}^n$ where 
$I\mapsto \deg(m_I)$, the complex $\mbb{F}(\eta)$ is the 
minimal free resolution of $R/N$.  
\end{proposition}

\begin{proof}
Suppose that $N$ is a Scarf ideal, set $L_N$ as the 
LCM-lattice of $N$ and let $\mbb{F}$ denote the minimal 
free resolution of $R/N$ with differential $\partial^\mbb{F}$.  
Fix a homological degree $p>0$ and an $I\subseteq\{1,\ldots,r\}$.  
For every $J\subset{I}$, the monomial $m_J<m_I$ in $L_N$.  
Supposing that $N$ is not lattice-linear, there exists 
$J=\{a_1,\ldots,\widehat{a_j},\ldots,a_p\}$ so that the 
coefficient of $e_J$ in the expansion of $\partial^\mbb{F}(e_I)$ 
is nonzero, and yet $m_J$ is not covered by $m_I$ in the 
lattice $L_N$.  Thus, there exists $m\in{L_N}$ so that $m_J<m<m_I$.  
Since $N$ is Scarf, $I$ and $J$ are uniquely determined from 
$m_I$ and $m_J$.  By definition $m=\lcm(m_{a_1},\ldots,m_{a_t})$ 
for some $\{a_1,\ldots,a_t\}$, and it follows that 
$J\subset\{a_1,\ldots,a_t\}\subseteq{I}=J\cup\{a_j\}$, forcing 
$\{a_1,\ldots,a_t\}=I$.  Therefore, $m=m_I$, and $m_J\ld{m_I}$.  
Hence $N$ is lattice-linear.
\end{proof}

\mysection{Proof of Theorem \ref{LLT}}\label{POMT}

\begin{proof}
It is clear that if $\mbb{F}(\eta)$ is the 
minimal free resolution then $N$ is lattice-linear.
It remains to show that lattice linearity implies 
that $\mbb{F}(\eta)$ is a resolution of $R/N$.    
We remark that since $L_N$ is a lattice, its set of atoms 
$G(N)$ forms a crosscut, 
%thus by Theorem \ref{CrossCut} 
so that Corollary \ref{LatticeCor} implies 
$C_\bullet(L_N,\Gamma)=C_\bullet(L_N)$.  We will use 
$C_\bullet(L_N,\Gamma)$ for our computation 
of $\mbb{F}(\eta)$.

Suppose that $N$ is a lattice-linear monomial ideal with minimal 
free resolution $\mbb{F}$ and let $B_i$ be a basis for $F_i$ 
as in Definition \ref{LLDef}.  With this choice of basis, let 
$F_{i,\alpha}$ be the free submodule of $F_i$ 
spanned by the set $$B_{i,\alpha}=\{v\in{B_i}:\deg(v)=\alpha\}.$$  
Hence, $$F_i=\bigoplus_{\alpha\in L_N}F_{i,\alpha},$$ 
and in particular 
$$(\mbb{F}_\alpha)_i=\bigoplus_{\beta\le\alpha}V_{i,\beta}$$ 
where $$V_{i,\beta}=\Bbbk\langle{v:v\in{B_{i,\beta}}}\rangle$$  
and $x^{\alpha-\beta}V_{i,\beta}$ is identified with $V_{i,\beta}$.

%Since $(\mbb{F}_\alpha)_i$ is the degree $\alpha$ strand of $\mbb{F}$ in homological degree $i$ and $\mbb{F}_{i,\alpha}$ is the degree $\alpha$ summand of $\mbb{F}_i$, there is a natural inclusion $\mbb{F}_{i,\alpha}\subset(\mbb{F_\alpha})_i$.  Appealing to the multigrading, we view $\mbb{F}_\beta$ as a subcomplex of $\mbb{F}_\alpha$.  

Making use of  $\mbb{T}$, the Taylor resolution of $R/N$, 
%the multigraded Betti numbers of $R/N$ are given by
%\begin{eqnarray*}
%\beta_{i,\alpha}(R/N) & = & \dim_\Bbbk\Tor_{i}^{R}(R/N,\Bbbk)_\alpha \\
%& = & \dim_{\Bbbk}\HH_{i}(\mbb{T}_\alpha/(\m\mbb{T})_\alpha) \\
%& = & \dim_\Bbbk\HH_{i}\left(\mbb{T}_\alpha\Big/\ds \sum_{\beta\ld\alpha}{\mbb{T}_\beta}\right). \\
%\end{eqnarray*}  
consider the exact sequence 
\begin{equation}\label{tayses}
0
\lra
{\ds \sum_{\beta\ld\alpha}{\mbb{T}_\beta}}
\lra
{\mbb{T}_\alpha}\lra{\mbb{T}_\alpha\Big/ \ds \sum_{\beta\ld\alpha}{\mbb{T}_\beta}}
\lra
0.
\end{equation}
The exactness of the Taylor resolution implies 
that $\mbb{T}_\alpha$ is an exact complex of 
vector spaces for $\hat{0}\ne\alpha\in{L_N}$.  
Indeed, $\mbb{T}_\alpha$ is acyclic with 
$\HH_0(\mbb{T}_\alpha)\cong(R/N)_\alpha$ and 
$(R/N)_\alpha=0$ for $x^\alpha\in{N}$.

Passing from (\ref{tayses}) to the long exact sequence in homology, 
the connecting homomorphism yields an isomorphism 
\begin{displaymath}
\begin{CD}
\HH_{i}\left(\mbb{T}_\alpha\Big/\ds \sum_{\beta\ld\alpha}{\mbb{T}_\beta}\right)
@>\mu_i>\cong>
\HH_{i-1}\left(\ds \sum_{\beta\ld\alpha}{\mbb{T}_\beta}\right),
\end{CD}
\end{displaymath}
which takes the class 
$$\ds [\bar{v}]\in\HH_{i}\left(\mbb{T}_\alpha\Big/\sum_{\beta\ld\alpha}{\mbb{T}_\beta}\right)$$ 
to the class 
$$\ds [\partial^\mbb{T}(v)]\in\HH_{i-1}\left(\sum_{\beta\ld\alpha}{\mbb{T}_\beta}\right)$$  
whenever $\bar{v}$ is a cycle in 
$\ds\mbb{T}_\alpha\Big/\sum_{\beta\ld\alpha}{\mbb{T}_\beta}$ 
represented by an element $v\in\mbb{T}_\alpha$.

Since $\mbb{F}$ is the minimal free resolution of $R/N$, we make 
the identifications 
\begin{eqnarray*} 
\ds\HH_i\left(\mbb{F}_\alpha\Big/(\m\mbb{F})_\alpha\right)
&=& \ds\left(\mbb{F}_\alpha/(\m\mbb{F})_\alpha\right)\\
&=& \ds\left(\mbb{F}_\alpha\Big/\sum_{\beta\ld\alpha}{\mbb{F}_\beta}\right)_i \\
&=& \ds(\mbb{F}_\alpha)_i\Big/\sum_{\beta\ld\alpha}(\mbb{F}_\beta)_i \\
&=& \ds\bigoplus_{\beta\le\alpha}V_{i,\beta}\Big/\bigoplus_{\beta<\alpha}V_{i,\beta}\\
&=& V_{i,\alpha}.
\end{eqnarray*}
Fixing an embedding of $\mbb{F}$ as a direct summand of $\mbb{T}$, 
we have $\mbb{T}=\mbb{F}\bigoplus\mbb{G}$ for some split exact complex 
of multigraded free modules $\mbb{G}$, and in particular, 
$\mbb{T}_\alpha=\mbb{F}_\alpha\bigoplus\mbb{G}_\alpha$ for every 
$\alpha$.  Since the induced map of complexes 
$$
\mbb{F}_\alpha\Big/\ds \sum_{\beta\ld\alpha}{\mbb{F}_\beta}
\lra
\mbb{T}_\alpha\Big/\ds \sum_{\beta\ld\alpha}{\mbb{T}_\beta}
$$ 
is split inclusion and is an isomorphism in homology, we 
consider $V_{i,\alpha}$ as a subspace of 
$\Z_i\left(\mbb{T}_\alpha\Big/
\ds \sum_{\beta\ld\alpha}{\mbb{T}_\beta}\right)$ 
and obtain the canonical identification 
$$
\ds \Z_i\left(\mbb{T}_\alpha\Big/\ds \sum_{\beta\ld\alpha}{\mbb{T}_\beta}\right)
=V_{i,\alpha}\bigoplus\B_i\left(\mbb{T}_\alpha\Big/\ds \sum_{\beta\ld\alpha}{\mbb{T}_\beta}\right).
$$  

Recalling the definitions of $\Gamma_\alpha$ 
and $\Gamma_{\alpha,\gamma}$ we see that
$$
\ds \sum_{\beta\ld\alpha}\mbb{T}_\beta
=\widetilde{\mathcal{C}}_\bullet(\Gamma_\alpha,\Bbbk)
$$
and
$$
\ds \sum_{\substack{\beta\ld\alpha \\ \gamma\ne\beta}}
\mbb{T}_{\gamma\wedge\beta}
=\widetilde{\mathcal{C}}_\bullet(\Gamma_{\alpha,\gamma},\Bbbk).
$$
Using these identifications, we have the following diagram
for each $\gamma\ld\alpha$ and each $i\ge 2$: 
\begin{equation}\label{bigCD}
\begin{CD} 
V_{i,\alpha}
@>\partial^{\mbb{F}}\circ{i}>> 
\left(\ds \sum_{\beta\ld\alpha}{\mbb{F}_\beta}\right)_{i-1} 
@>\pr_\gamma>> 
V_{i-1,\gamma}\\
@V\textup{incl} VV @. @VV\textup{incl} V \\
\Z_i\left(\mbb{T}_\alpha\Big/\ds \sum_{\beta\ld\alpha}{\mbb{T}_\beta}\right)
@.
@.
\Z_{i-1}\left(\mbb{T}_\gamma\Big/\ds \sum_{\nu\ld\gamma}{\mbb{T}_\nu}\right)\\
@V \pr VV @. @VV\pr V \\
\HH_{i}\left(\mbb{T}_\alpha\Big/\ds \sum_{\beta\ld\alpha}{\mbb{T}_\beta}\right)
@.
@.
\HH_{i-1}\left(\mbb{T}_\gamma\Big/\ds \sum_{\nu\ld\gamma}{{\mbb{T}_\nu}}\right)\\
@V\mu_{i}V\cong V @. @V\cong V\mu_{i-1} V \\
\HH_{i-1}\left(\ds \sum_{\beta\ld\alpha}{\mbb{T}_\beta}\right)
@>\delta^{\alpha,\gamma}_{i-1}>>
\HH_{i-2}\left(\ds \sum_{\substack{\beta\ld\alpha \\ \gamma\ne\beta}}{\mbb{T}_{\gamma\wedge\beta}}\right)
@>\iota>>
\HH_{i-2}\left(\ds \sum_{\nu\ld\gamma}{\mbb{T}_\nu}\right)\\
@| @| @| \\
\Ho_{i-2}(\Gamma_{\alpha},\Bbbk)
@>\delta^{\alpha,\gamma}_{i-1}>>
\Ho_{i-3}(\Gamma_{{\alpha,\gamma}},\Bbbk)
@>\iota>>
\Ho_{i-3}(\Gamma_{\gamma},\Bbbk)
\end{CD}
\end{equation}
We claim this diagram is commutative.

Let $v\in{V_{i,\alpha}}$ so that by the assumption of 
lattice-linearity,
$$
\partial^{\mbb{F}}(v)
=\sum_{\beta\ld\alpha}v_\beta
$$
where each $v_\beta\in V_{i-1,\beta}$.  
Canonically, $\textup{incl}(v)=v$.  
Under projection, the cycle $v$ is sent to its corresponding 
class in homology, $[v]$.  As mentioned above, the connecting 
map $\mu_i$ is an isomorphism, and 
$$
\mu_i([v])
=[\partial^\mbb{T}(v)]
=[\partial^\mbb{F}(v)]
=\left[\sum_{\beta\ld\alpha}v_\beta\right].
$$  
Applying $\delta^{\alpha,\gamma}_{i-1}$, which is the 
connecting Mayer-Vietoris map, 
$$
\delta^{\alpha,\gamma}_{i-1}\left(\left[\sum_{\beta\ld\alpha}v_\beta\right]\right)
=[\partial^\mbb{T}(v_\gamma)].
$$  
Lastly, $\iota$ is the homological inclusion map and thus, 
$\iota([\partial^\mbb{T}(v_\gamma)])=[\partial^\mbb{T}(v_\gamma)]$.  

Again taking $v\in{V_{i,\alpha}}$, we appeal to the differential 
of $\mbb{F}$, and obtain 
$$(\partial^\mbb{F}\circ{i})(v)=\partial^\mbb{F}(v)=\sum_{\beta\ld\alpha}v_\beta.$$  
Projecting onto $V_{i-1,\gamma}$, we have 
$$\ds \pr_\gamma\left(\sum_{\beta\ld\alpha}v_\beta\right)=v_\gamma.$$  
The inclusion map now gives $\textup{incl}(v_\gamma)=v_\gamma$, and 
passing this cycle to homology yields $[v_\gamma]$.  Through the 
isomorphism, $\mu_{i-1}([v_\gamma])=[\partial^\mbb{T}(v_\gamma)]$, 
which completes the proof of the commutativity of the diagram.

We complete the proof of the Theorem by establishing the connection between 
lattice linearity and the poset construction.  
%Indeed, given 
%the acyclic complex $C_\bullet(L_N)$ and $i>0$, set 
%$$\ds (C_\alpha)_i=\bigoplus_{\beta\le\alpha}C_{i,\beta}$$ 
%and define the map $d_i:(C_\alpha)_i\lra(C_\alpha)_{i-1}$ as 
%$$\ds d_i(w)=\sum_{\beta\le\alpha}d_{\beta,i}(w)=\sum_{\beta\ld\alpha}d_{\beta,i}(w)$$ 
%for $w\in{(C_\alpha)_i}$.  

For each $i\ge 0$ define the isomorphism of free $R$-modules
$\psi_i:F_i\lra\mbb{F}(\eta)_i$ on a basis 
element $v\in V_{i,\alpha}\subset B_i$ by 
applying the left column of (\ref{bigCD}), thus 
$$
\psi_i(v)
=\otimes[\mu_i\circ\pr\circ\textup{incl}(v)]
=1\otimes[\partial^\mbb{F}(v)]
\in{R\otimes C_{i,\alpha}}
\subset\mbb{F}(\eta)_i.
$$ 

%which maps each $v\in{V_{i,\beta}}$ to $[\partial^\mbb{F}(v)]\in{C_{i,\beta}}$ for every $\beta\le\alpha$. 
%Fixing $\beta\le\alpha$ and choosing $v\in{V_{i,\beta}}$ then by 
%definition, $\psi_i(v)=[\partial^\mbb{F}(v)]$.  Under the map 
%$\partial_i$, the class $[\partial^\mbb{F}(v)]$ is taken to 
%$\ds \sum_{\deg(v')\ld\beta}[\partial^\mbb{F}(v')]$.  Using 
%the differential of $\mbb{F}$, the element $v$ is mapped to 
%$\ds \partial^\mbb{F}(v)=\sum_{\deg(v')\ld\beta}v'$, which under 
%$\psi_{i-1}$ is sent to  $\ds \sum_{\deg(v')\ld\beta}[\partial^\mbb{F}(v')]$.  

By the commutativity of (\ref{bigCD}), 
we have a commutative diagram,
\begin{equation}\label{mfrDiag}
\begin{CD}
(\mbb{F})_i@>\partial^\mbb{F} >> (\mbb{F})_{i-1}\\
@V\psi_i VV  @VV\psi_{i-1}V\\
(\mbb{F}(\eta))_i @>> d_i > (\mbb{F}(\eta))_{i-1}
\end{CD}
\end{equation}
for every $i\ge 2$.  Furthermore, (\ref{mfrDiag}) commutes trivially 
for $i=1$.  It follows that the sequences $\mbb{F}$ 
and $\mbb{F}(\eta)$ are isomorphic, hence 
$\mbb{F}(\eta)$ is the minimal free resolution of $R/N$.
\end{proof}
\begin{acknowledgments}
I would like to thank my thesis advisor, Alexandre Tchernev for helpful 
discussions.
\end{acknowledgments}
\mysection{Appendix: Properties of $C_\bullet(P)$, by Alexandre Tchernev}\label{AppendixTchernev}
{\small {\sc Alexandre Tchernev, Department of Mathematics and Statistics, 
University at Albany SUNY, Albany, NY 12222}}

\emph{E-mail address:} \quad {\tt tchernev@math.albany.edu}

\bigskip
We present with proofs some additional properties of the sequence $C_\bullet(P)$ 
as constructed in Section \ref{Comb}.  The same properties hold for 
the sequence $C_\bullet(P,\Gamma)$ with similar proofs which we leave 
to the reader as an exercise.  

\begin{proposition}\label{itsacomplex}
If $P$ is a ranked poset 
 then $C_\bullet(P)$ is a complex of vector spaces.
\end{proposition}

We begin by establishing notation for some of the relevant objects.  

Set $\Delta_{\alpha}^{(0)}=\D_{\alpha}$, and for $j\ge{1}$, write 
$$
\Delta_{\alpha}^{(j)}
=\bigcup_{\substack{\beta<\alpha \\ \rk(\beta)=\rk(\alpha)-j}}
\Delta_{\beta}^{(0)}.
$$  
Using the inclusion $\Delta_{\alpha}^{(1)}\subset\Delta_{\alpha}^{(0)}$, 
and the fact that $\Delta_{\alpha}^{(0)}$ is contractible, 
we obtain a canonical isomorphism 
\begin{equation}\label{Isom011}
\begin{CD}
\Ho_{i}(\Delta_{\alpha}^{(0)}/\Delta_{\alpha}^{(1)},\Bbbk)
@> \theta_{i,\alpha} > \cong > 
\Ho_{i-1}(\Delta_{\alpha}^{(1)},\Bbbk)\\
\end{CD}
\end{equation}
for every $i$ using the long exact sequence in relative homology.  
Further, the equality of reduced chain complexes
$$
\CC\left(\Delta_\alpha^{(j)}/\Delta_\alpha^{(j+1)}\right)=
\bigoplus_{\stackrel{\beta<\alpha}{\rk(\beta)=\rk(\alpha)-j}}\CC\left(\Delta_\beta^{(0)}/\Delta_\beta^{(1)}\right)
$$
for each $j$ gives rise to an isomorphism on the level of reduced homology,
$$
\Ho_*\left(\Delta_\alpha^{(j)}/\Delta_\alpha^{(j+1)}\right)
=
\bigoplus_{\stackrel{\beta<\alpha}{\rk(\beta)=\rk(\alpha)-j}}\Ho_*\left(\Delta_\beta^{(0)}/\Delta_\beta^{(1)}\right)
$$
which we refer to as \emph{reindexing}.

Through the inclusion of simplicial complexes 
$\Delta_{\alpha}^{(2)}\subset\Delta_{\alpha}^{(1)}$, 
we obtain an exact sequence of reduced chain complexes,
\begin{displaymath}
\begin{CD}
0
\lra
\CC(\Delta_{\alpha}^{(2)})
\lra
\CC(\Delta_{\alpha}^{(1)}) 
@>>> 
\CC(\Delta_{\alpha}^{(1)}/\Delta_{\alpha}^{(2)}) 
@>>> 0 \\
@. @| \\
@. \ds \bigoplus_{\beta\ld\alpha}
   \CC(\Delta_{\beta}^{(0)}/\Delta_{\beta}^{(1)}) @.\\
\end{CD}  
\end{displaymath}
which in turn gives rise to a long exact sequence in reduced homology,
\begin{displaymath}
\begin{CD}
\cdots\lra
\Ho_{i}(\Delta_{\alpha}^{(2)},\Bbbk)\lra\Ho_{i}(\Delta_{\alpha}^{(1)},\Bbbk) @>>> \Ho_{i}(\Delta_{\alpha}^{(1)}/\Delta_{\alpha}^{(2)},\Bbbk) @>>> \cdots \\
@. @| \\
@. \ds \bigoplus_{\beta\ld\alpha}\Ho_i(\Delta_{\beta}^{(0)}/\Delta_{\beta}^{(1)},\Bbbk) @.\\
@. @V\cong V\oplus\theta_{i,\beta}V \\
@. \ds \bigoplus_{\beta\ld\alpha}\Ho_{i-1}(\Delta_{\beta}^{(1)},\Bbbk). @.\\
\end{CD}  
\end{displaymath}
The simplicial inclusions 
$\Delta_{\alpha}^{(3)}\subset\Delta_{\alpha}^{(2)}\subset\Delta_{\alpha}^{(1)}$ 
also give rise to an exact sequence of relative chain complexes 
$$
0
\ra
\CC(\Delta_{\alpha}^{(2)}/\Delta_{\alpha}^{(3)},\Bbbk)
\ra
\CC(\Delta_{\alpha}^{(1)}/\Delta_{\alpha}^{(3)},\Bbbk) 
\ra
\CC(\Delta_{\alpha}^{(1)}/\Delta_{\alpha}^{(2)},\Bbbk) 
\ra
0 
$$
which produces a long exact sequence in reduced homology
\begin{equation}\label{Quotients}
\cdots
\ra
\Ho_i(\Delta_{\alpha}^{(1)}/\Delta_{\alpha}^{(3)},\Bbbk)
\stackrel{\mu}{\ra}
\Ho_i(\Delta_{\alpha}^{(1)}/\Delta_{\alpha}^{(2)},\Bbbk)
\stackrel{D}{\ra}
\Ho_{i-1}(\Delta_{\alpha}^{(2)}/\Delta_{\alpha}^{(3)},\Bbbk)
\ra
\cdots.
\end{equation}

Finally, we have the equality  
$$
\ds
\Ho_{i-1}(\Delta_{\alpha}^{(2)}/\Delta_{\alpha}^{(3)},\Bbbk)=
\bigoplus_{\stackrel{\gamma<\alpha}{\rk(\gamma)=\rk(\alpha)-2}}
  \Ho_{i-1}(\Delta_{\gamma}^{(0)}/\Delta_{\gamma}^{(1)},\Bbbk),
$$ 
which is given by reindexing.

We can now proceed with the proof of Proposition \ref{itsacomplex}, 
which we break into three lemmas.  

\begin{lemma}\label{Phi21}
We have $\varphi_1\circ\varphi_2=0$.
\end{lemma}

\begin{proof} 
Suppose $[w]\in \Ho_0(\Delta_{\alpha}^{(1)},\Bbbk)$, 
with representative cycle $w$.
We therefore have
$$
w=\sum_{\lambda\in\left(\hat{0},\alpha\right)} c_\lambda\cdot\{\lambda\}
$$  
with $\ds \sum_{\lambda\in\left(\hat{0},\alpha\right)} c_\lambda=0$.
Choosing a partition of $\left(\hat{0},\alpha\right)$ into a disjoint 
union 
\begin{equation}\label{partition}
\left(\hat{0},\alpha\right)=\bigsqcup_{\beta\ld\alpha}P_\beta
\end{equation}
of subsets $P_\beta$ such that for every $\lambda\in P_\beta$
one has $\lambda\le\beta$, we get
$$
w=\sum_{\beta\ld\alpha} w_\beta
$$
with
$$
w_\beta
=\sum_{\lambda\in P_\beta} c_\lambda\cdot\{\lambda\}.
$$
Therefore, 
$$
\varphi_2^{\alpha,\beta}([w])
=[d(w_\beta)]
=\left[\sum_{\lambda\in P_\beta} c_\lambda\cdot\emptyset\right]
$$
where $d$ is the usual boundary map.

Applying $\varphi_1$, we now have 
\begin{eqnarray*}
\varphi_1\circ\varphi_2([w]) & 
= & \sum_{\beta\ld\alpha}\varphi_1\circ\varphi_2^{\alpha,\beta}([w])\\
& = & \sum_{\beta\ld\alpha}
 \varphi_1\left(\left[\sum_{\lambda\in P_\beta} c_\lambda\cdot\emptyset\right]\right)\\
& = & \sum_{\beta\ld\alpha}
 \left[\sum_{\lambda\in P_\beta} c_\lambda\cdot\emptyset\right]\\
& = & \sum_{\beta\ld\alpha}
 \sum_{\lambda\in P_\beta} c_\lambda\cdot\emptyset\\
& = & \sum_{\lambda\in \left(\hat{0},\alpha\right)} c_\lambda\cdot\emptyset\\
& = & 0.
\end{eqnarray*}
\end{proof}

\begin{lemma}\label{phiCD}
For each $i\ge 1$ the diagram  
\begin{equation*}
\xymatrix{ 
\Ho_{i}(\Delta_{\alpha}^{(1)}/\Delta_{\alpha}^{(2)},\Bbbk) \ar[r]^D \ar@{=}[d]
& \Ho_{i-1}(\Delta_{\alpha}^{(2)}/\Delta_{\alpha}^{(3)},\Bbbk) \ar@{=}[d]\\
\ds \bigoplus_{\beta\ld\alpha}\Ho_{i}\left(\Delta_{\beta}^{(0)}/\Delta_{\beta}^{(1)},\Bbbk\right)
    \ar[d]^\cong_{\oplus\theta_{i,\beta}}
& \ds \bigoplus_{\stackrel{\gamma<\alpha}{\rk(\gamma)=\rk(\alpha)-2}}
\Ho_{i-1}(\Delta_{\gamma}^{(0)}/\Delta_{\gamma}^{(1)},\Bbbk) \ar[d]_\cong^{\oplus\theta_{i-1,\gamma}} \\
 \ds \bigoplus_{\beta\ld\alpha}\Ho_{i-1}(\Delta_{\beta}^{(1)},\Bbbk) \ar[r]^{\varphi_{i+1}}
& \ds \bigoplus_{\stackrel{\gamma<\alpha}{\rk(\gamma)=\rk(\alpha)-2}}\Ho_{i-2}(\Delta_{\gamma}^{(1)},\Bbbk)}
\end{equation*}
is commutative. 
\end{lemma}

\begin{proof}
To verify commutativity, it suffices to show 
that for each $\delta<\alpha$ with $\rk(\delta)=\rk(\alpha)-2$ the 
components of 
$\ds\varphi_i\circ\left(\oplus_{\beta\ld\alpha}\theta_{i,\beta}\right)$ 
and $\ds\left(\oplus_{\beta\ld\alpha}\theta_{i-1,\gamma}\right)\circ D$ 
are the same in $\Ho_{i-2}(\Delta_{\delta}^{(1)},\Bbbk)$.  

Indeed, suppose that $[\bar{w}]$ is a representative 
for the homology class generated by the image $\bar{w}$ in
$\widetilde{C}_i\left(\Delta_\alpha^{(1)}/\Delta_\alpha^{(2)}\right)$
of the relative cycle
$$
w
=\sum_{a_i^\sigma\ld\alpha} c_\sigma\cdot\sigma
=\sum_{\beta\ld\alpha} w_\beta
$$ 
of $\left(\Delta_\alpha^{(1)},\Delta_\alpha^{(2)}\right)$,
where $c_\sigma\in\Bbbk$, each face   
$\sigma=\{a^\sigma_0,\ldots,a^\sigma_i\}$ is oriented by 
$a^\sigma_0<\cdots<a^\sigma_i$ 
and $\ds w_\beta=\sum_{a_i^\sigma=\beta}c_\sigma\cdot\sigma$. 
Since $w$ is a relative cycle, we must have 
$$
\sum_{a_i^\sigma=\beta} c_\sigma\cdot d(\hat{\sigma})=0
$$ 
where $\hat{\sigma}=\{a^\sigma_0,\ldots,a^\sigma_{i-1}\}$
and $d$ is the usual boundary map.
Therefore, each $w_\beta$ is a relative cycle for 
$(\Delta_\beta^{(0)},\Delta_\beta^{(1)})$
and $[\bar{w}]$ has 
$$
[\bar{w}]
=\sum_{\beta\ld\alpha}[\bar{w}_\beta]
$$
as its reindexing decomposition.  Thus,
$$
\left(\oplus\theta_{i,\beta}\right)[\bar{w}]
=\sum_{\beta\ld\alpha}\theta_{i,\beta}([\bar{w}])
=\sum_{\beta\ld\alpha}[d(w_\beta)]
=\sum_{\beta\ld\alpha}[v_\beta]
$$
where 
$$
v_\beta
=d(w_\beta)
=(-1)^i\sum_{a_i^\sigma=\beta}c_\sigma\cdot\hat{\sigma}
$$
is a cycle in $\Delta_\beta^{(0)}$.

Next, choose for each 
$\beta\ld\alpha$ a partition 
$$
\left(\hat{0},\beta\right)
=\bigsqcup_{\gamma\ld\beta}P_{\beta,\gamma}
$$
of $\left(\hat{0},\beta\right)$ such that 
$\lambda\le\gamma$ for each 
$\lambda\in P_{\beta,\gamma}$ and write 
$$
v_\beta
=\sum_{\gamma\ld\beta}w_{\beta,\gamma}
$$
where 
$$
w_{\beta,\gamma}
=(-1)^i\sum_{\stackrel{a_i^\sigma=\beta}{a_{i-1}^\sigma\in P_{\beta,\gamma}}}
 c_\sigma\cdot\hat{\sigma}.
$$
It follows that the component of 
$\varphi_{i+1}\left[\left(\oplus\theta_{i,\beta}\right)([\bar{w}])\right]$
in $\Ho_{i-2}(\Delta_{\delta}^{(1)},\Bbbk)$ is given by 
$$
\sum_{\beta:\delta<\beta<\alpha}
   \varphi_{i+1}^{\beta,\delta}\left([v_\beta]\right)
=\sum_{\beta:\delta\ld\beta\ld\alpha}
   \varphi_{i+1}^{\beta,\delta}\left(\left[\sum_{\gamma\ld\beta}w_{\beta,\gamma}\right]\right)
=\sum_{\beta:\delta\ld\beta\ld\alpha}
   \left[d(w_{\beta,\delta})\right].  
$$
On the other hand, since $v_\beta$ is a cycle, 
each $w_{\beta,\gamma}$ is a relative cycle for 
$(\Delta_\alpha^{(2)},\Delta_\alpha^{(3)})$.
As $D$ is the connecting map in (\ref{phiCD}), we have
\begin{eqnarray*}
D([\bar{w}]) & = & \left[\overline{d(w)}\right]\\
& = & \left[\overline{\sum_{\beta\ld\alpha}d(w_\beta)}\right]\\
& = & \left[\sum_{\beta\ld\alpha}\bar{v}_\beta\right]\\
& = & \left[\sum_{\beta\ld\alpha}\left(\sum_{\gamma\ld\beta}w_{\beta,\gamma}\right)\right]\\
& = & \sum_{\stackrel{\gamma<\alpha}{\rk(\gamma)=\rk(\alpha)-2}}
      \left(\sum_{\gamma\ld\beta\ld\alpha}[\bar{w}_{\beta,\gamma}]\right)\\
& = & \sum_{\stackrel{\gamma<\alpha}{\rk(\gamma)=\rk(\alpha)-2}}
      \left[\sum_{\gamma\ld\beta\ld\alpha}\bar{w}_{\beta,\gamma}\right]\\
\end{eqnarray*}
as its reindexing decomposition.
Therefore, the component  of 
$(\oplus_\gamma\theta_{i-1,\gamma})\circ(D([\bar{w}]))$
in $\Ho_{i-2}(\Delta_{\delta}^{(1)},\Bbbk)$ is equal to 
$$
\theta_{i-1,\delta}
 \left(
 \sum_{\beta:\delta\ld\beta\ld\alpha}\bar{w}_{\beta,\delta}
 \right)
=\left[d\left(\sum_{\beta:\delta\ld\beta\ld\alpha}w_{\beta,\delta}\right)\right] 
=\sum_{\beta:\delta\ld\beta\ld\alpha}[d(w_{\beta,\delta})],
$$
which proves commutativity. 
\end{proof} 

\begin{lemma}\label{smallCD}
For each $i\ge 1$ the diagram 
\begin{equation}
\xymatrix{ 
\Ho_{i}(\Delta_{\alpha}^{(1)}/\Delta_{\alpha}^{(3)},\Bbbk) \ar[r]^\mu
& \Ho_{i}(\Delta_{\alpha}^{(1)}/\Delta_{\alpha}^{(2)},\Bbbk) \ar@{=}[d]\\
& \ds \bigoplus_{\beta\ld\alpha}\Ho_{i}\left(\Delta_{\beta}^{(0)}/\Delta_{\beta}^{(1)},\Bbbk\right)
    \ar[d]_\cong^{\oplus\theta_{i,\beta}}\\
\Ho_{i}(\Delta_{\alpha}^{(1)},\Bbbk)\ar[r]^{\varphi_{i+2}} \ar[uu]^\pi
& \ds \bigoplus_{\beta\ld\alpha}\Ho_{i-1}(\Delta_{\beta}^{(1)},\Bbbk),}
\end{equation}
is commutative.  
\end{lemma}

\begin{proof}
It is enough to show that $\varphi_{i+2}$ and 
$(\oplus\theta_{i,\beta})\circ D\circ\mu\circ\pi$
have the same component in $\Ho_{i-1}(\Delta_{\beta}^{(1)},\Bbbk)$.

Suppose that $[v]$ is a 
homology class in $\Ho_{i}(\Delta_{\alpha}^{(1)},\Bbbk)$ 
represented the by the cycle 
$$
v=\ds\sum_{a_i^\sigma<\alpha}c_\sigma\cdot\sigma.
$$
Under projection 
$$
\pi([v])
=[\pi(v)]
=\left[\ds\sum_{a_i^\sigma<\alpha}c_\sigma\cdot\sigma'\right],
$$
where $\sigma'=\pi(\sigma)$ is the image of 
$\sigma$ under the standard projection 
$\CC(\Delta_\alpha^{(1)})\lra\CC(\Delta_\alpha^{(1)}/\Delta_\alpha^{(3)})$.  
Applying $\mu$ to $\pi([v])$, we have 
$$
\mu([\pi(v)])
=\left[\ds\sum_{a_i^\sigma<\alpha}c_\sigma\cdot\bar{\sigma}\right]
=\ds\sum_{\beta\ld\alpha}[\bar{w}_\beta],
$$
where $\bar{\sigma}$ is the image of $\sigma$ 
under the projection 
$\CC(\Delta_\alpha^{(1)})\lra\CC(\Delta_\alpha^{(1)}/\Delta_\alpha^{(2)})$,   
the process of reindexing has produced   
$$
w_\beta
=\ds\sum_{a_i^\sigma\in P_\beta}c_\sigma\cdot\sigma,
$$
and $P_\beta$ is as in (\ref{partition}).
Thus 
$$
\left(\oplus\theta_{i,\beta}\right)\left(\ds\sum_{\beta\ld\alpha}[\bar{w}_\beta]\right)
=\sum_{\beta\ld\alpha}\theta_{i,\beta}([\bar{w}_\beta])
=\sum_{\beta\ld\alpha}[d(w_\beta)]
=\sum_{\beta\ld\alpha}[v_\beta]
$$
where 
$$
v_\beta
=d(w_\beta)
=(-1)^i\left(\sum_{a_i^\sigma=\beta}c_\sigma\cdot\hat{\sigma}\right)
 +d\left(\sum_{a_i^\sigma\in P_\beta\smsm\{\beta\}}c_\sigma\cdot\sigma\right)
$$
and  
$\hat{\sigma}=\sigma\smsm\{a_i^\sigma\}$.  
Clearly, $[v_\beta]$, is the component 
of $(\oplus\theta_{i,\beta})\circ D\circ\mu\circ\pi$ in 
$\Ho_{i-1}(\Delta_{\beta}^{(1)},\Bbbk)$.  

Next, given $\beta\ld\alpha$ we have 
$$
v
=\ds\sum_{a_i^\sigma<\alpha}c_\sigma\cdot\sigma
=\ds\sum_{\beta\ld\alpha}\left(\sum_{a_i^\sigma\in P_\beta}c_\sigma\cdot\sigma\right)
$$
and therefore, 
\begin{eqnarray*}
\varphi_{i+2}([v]) & = & \ds\sum_{\beta\ld\alpha}\varphi_{i+2}^{\alpha,\beta}([v])\\
& = & \ds\sum_{\beta\ld\alpha}\left[d\left(\ds\sum_{a_i^\sigma\in P_\beta}c_\sigma\cdot\sigma\right)\right]\\
& = & \ds\sum_{\beta\ld\alpha}\left[(-1)^i\left(\sum_{a_i^\sigma=\beta}c_\sigma\cdot\hat{\sigma}\right)
 +d\left(\sum_{a_i^\sigma\in P_\beta\smsm\{\beta\}}c_\sigma\cdot\sigma\right)\right]\\
& = & \sum_{\beta\ld\alpha}[v_\beta]
\end{eqnarray*}
so that the commutativity has been proved.  
\end{proof}

\begin{proof}[Proof of Proposition \ref{itsacomplex}]
We show $\varphi_{i-1}\circ\varphi_i=0$.  Lemma \ref{Phi21} 
establishes this for $i=2$, and 
the commutative diagrams of Lemmas \ref{phiCD} and \ref{smallCD} may 
be combined so that for $i\ge 3$ we have the equality  
$\varphi_{i-1}\circ\varphi_i=(\oplus\theta_{i-1,\gamma})\circ D\circ\mu\circ\pi$.  
Since $D$ and $\mu$ are consecutive maps 
in an exact sequence, $D\circ\mu=0$.  
Therefore $\varphi_{i-1}\circ\varphi_i=0$ for each $i$ 
which completes the proof that $C_\bullet(P)$ is a complex.
\end{proof}

The general case can always be reduced to the case that 
$P$ is ranked because of the following.  

\begin{lemma}\label{rankcanon}
There is a rank preserving canonical embedding
$P\subset P'$, where $P'$ is a ranked poset, so that 
for every $i\ge{0}$ and $\alpha\in P'$ one has 
$$
C_{i,\alpha}(P')
=\left\{
\begin{array}{ll}
C_{i,\alpha}(P) & \textup{if}\quad \alpha\in P,\\
0 & \textup{otherwise.}
\end{array}
\right.
$$
In particular, $C_{i}(P)=C_{i}(P')$ for every $i\ge{0}$.
\end{lemma}

\begin{proof}
For $\beta\ld\alpha$ with $\rk_P(\alpha)-\rk_P(\beta)\ge 2$, let  
$$
P_{\alpha,\beta}=
\{\gamma_i^{\alpha,\beta}:1\le i\le \rk_P(\alpha)-\rk_P(\beta)-1\}
$$
be a set of symbols, and define $P'$ as the disjoint union 
$$
P'=P\bigsqcup\left(\bigsqcup_{\beta\ld\alpha\in P}P_{\alpha,\beta}\right).
$$
We order the elements of $P'$ by describing the covering relations:
the Hasse diagram of $P'$ is obtained from the Hasse diagram of $P$ 
by breaking up an edge $\beta\ld\alpha$ into $n$ edges $\beta\ld\gamma_1^{\alpha,\beta}\ld\cdots\ld\gamma_{n-1}^{\alpha,\beta}\ld\alpha$
where $n=\rk_P(\alpha)-\rk_P(\beta)$.  It is clear that in 
this way, $P'$ is canonically determined by $P$ and 
$\rk_P(\alpha)=\rk_{P'}(\alpha)$ for each $\alpha\in P$.  
For the remainder of the proof, we transfer this notation to any 
set (poset, simplicial complex) $X$ associated to $P$ and 
write $X'$ for the corresponding set (poset, simplicial complex).  

Since $P'$ can be obtained iteratively by breaking up one edge in 
two at a time, to prove the second claim of the Lemma it is enough 
to assume that only one additional poset element $\gamma$ is added 
to $P$ and $\beta\ld\gamma\ld\alpha$ in $P'$ for some 
$\beta\ld\alpha$ in $P$.  

%Tchernev deleted June 2nd.
\begin{comment}
The uniqueness of covering implies 
that $\cplx{\gamma}'=\sx{\beta}$ and hence 
$\Ho_i(\cplx{\gamma}',\Bbbk)=0$ for every $i$.  Further, 
there is an equality $\sx{\gamma}=\sx{\beta}$.  Indeed, taking 
$a\in{S}$ so that $a<\beta$, then $a<\gamma$ as $a<\beta\ld\gamma$.  
Conversely, every atom $a\in{S}$ which is comparable to $\gamma$ 
is also comparable to $\beta$, the unique element which $\gamma$ 
covers since any chain containing $a$ and $\gamma$ must also contain 
$\beta$.  Therefore, 
$$
\ds\cplx{\alpha}'
=\bigcup_{\rho\ld{\alpha}}\sx{\rho}
={\sx{\gamma}}\cup\bigg(\bigcup_{\gamma\ne\rho\ld\alpha}\sx{\rho}\bigg)
={\sx{\beta}}\cup\bigg(\bigcup_{\gamma\ne\rho\ld\alpta}}\cup\bigg(\bigcup_{\gamma\ne\rho\ld\alpha}\sx{\rho}\bigg)
=\cplx{\alpha}.
$$  
Since this equality exists on the level of simplicial 
complexes, then $\Ho_i(\cplx{\alpha},\Bbbk)$ is left 
unchanged for every $i$ and $\alpha\in{P}$ when the 
element $\gamma$ is introduced to $P$ and homological 
analysis is instead performed on the new poset $P'$.  
\end{comment}

We have 
$$
\ds \Delta_\gamma'
=\bigcup_{\rho\ld{\gamma}}\G_\rho'
=\G_\beta'
=\G_\beta
$$ 
since $\gamma$ uniquely covers $\beta$ by construction.  
Thus, $\Delta_\gamma'$ is a cone with apex $\beta$ and hence 
contractible.  Therefore, $C_{i,\gamma}(P')=0$ for each $i$.  
Next, let $\delta\in P$.  If $\delta\not\ge\alpha$ then 
$\Delta'_\delta=\Delta_\delta$ and so 
$C_{i,\delta}(P')=C_{i,\delta}(P)$ for each $i$.  Thus, 
it remains to consider the case $\delta\ge\alpha$.  

Let 
$$
\Omega
=\str{\Delta'_\delta}{\{\alpha\}}
=\G'_\beta*\Delta\left([\alpha,\delta)\right)
=\G_\beta*\Delta\left([\alpha,\delta)\right).
$$  
Thus $\Delta'_\delta=\Delta_\delta\cup\Omega$ and 
$\Omega\cap\Delta_\delta=\G_\beta*\Delta\left([\alpha,\delta)\right)$ 
is a cone with apex $\beta$, hence contractible.  
The Mayer-Vietoris sequence in reduced homology on the 
triple $\big(\Delta_\delta,\Omega,\Delta'_\delta\big)$ yields 

\begin{displaymath}
\begin{CD}
\vdots\\
@VVV \\
\Ho_{i}\big(\Delta_\delta\cap\Omega,\Bbbk\big)\\
@VVV \\
\Ho_{i}\big(\Delta_\delta,\Bbbk\big)\oplus\Ho_{i}\big(\Omega,\Bbbk\big)\\
@VVV \\
\Ho_{i}\big(\Delta_\delta',\Bbbk\big)\\
@VVV \\
\Ho_{i-1}\big(\Delta_\delta\cap\Omega,\Bbbk\big)\\
@VVV \\
\vdots
\end{CD}  
\end{displaymath}

Since both $\Omega$ and $\Omega\cap\Delta_\delta$ are 
contractible we get the desired conclusion.  
\end{proof}

\begin{remark}
Replacing the poset $P$ with the ranked poset $P'$ 
may result in different sequences $C_\bullet(P)$ and 
$C_\bullet(P')$.  Indeed, although the components 
are identical by Lemma \ref{rankcanon}, the maps 
in the sequences are different in general.  In particular, 
the maps present in $C_\bullet(P')$ will in general have
more trivial components.
\end{remark}

\end{document}